\documentclass{amsart}

\usepackage{amsmath}
\usepackage{amsthm}
\usepackage{amsfonts}

\newtheorem{lemma}{Lemma}[section]

\newtheorem{proposition}[lemma]{Proposition}
\newtheorem{theorem}[lemma]{Theorem}
\newtheorem{corollary}[lemma]{Corollary}
\newtheorem{definition}[lemma]{Definition}

\newtheorem{conjecture}[lemma]{Conjecture}

\newcommand{\tr}{\mathrm{tr}}
\newcommand{\R}{\mathbb{R}}

\newcommand{\Z}{\mathbb{Z}}

\newcommand{\sA}{\mathsf{A}}
\title[Lower spectral radius for low-rank matrices]{Generic properties of the lower spectral radius for some low-rank pairs of matrices}

\author{Ian D. Morris}

\begin{document}

\begin{abstract}

The lower spectral radius of a set of $d \times d$ matrices is defined to be the minimum possible exponential growth rate of long products of matrices drawn from that set. When considered as a function of a finite set of matrices of fixed cardinality it is known that the lower spectral radius can vary discontinuously as a function of the matrix entries. In a previous article the author and J. Bochi conjectured that when considered as a function on the set of all pairs of $2 \times 2$ real matrices, the lower spectral radius is discontinuous on a set of positive (eight-dimensional) Lebesgue measure, and related this result to an earlier conjecture of Bochi and Fayad. In this article we investigate the continuity of the lower spectral radius in a simplified context in which one of the two matrices is assumed to be of rank one. We show in particular that the set of discontinuities of the lower spectral radius on the set of pairs of $2 \times 2$ real matrices has positive seven-dimensional Lebesgue measure, and that among the pairs of matrices studied, the finiteness property for the lower spectral radius is true on a set of full Lebesgue measure but false on a residual set.\\ \\

\end{abstract}

\maketitle


\section{Introduction}
The joint spectral radius $\overline{\varrho}(\mathsf{A})$ of a set $\mathsf{A}$ of $d \times d$ matrices, introduced by G.-C. Rota and G. Strang in 1960 \cite{Ro03,RoSt60}, is defined to be the maximum possible exponential growth rate of a product of matrices drawn from the set $\mathsf{A}$:
\[\overline{\varrho}(\mathsf{A}):=\lim_{n \to \infty}\sup\left\{\|A_n \cdots A_1\|^{\frac{1}{n}} \colon A_i \in \mathsf{A}\right\}=\sup_{(A_i)_{i=1}^\infty \in \mathsf{A}^{\mathbb{N}}} \liminf_{n \to \infty}\|A_n \cdots A_1\|^{\frac{1}{n}}.\]
In the last two decades the joint spectral radius has been the subject of a proliferation of research articles which we will not attempt to list here exhaustively; some prominent recent contributions include \cite{BlNe05,BlThVl03,BoMa02,GuPr13,GuWiZe05,Gu95,HaMoSiTh11,Ju09,LaWa95}. By contrast the lower spectral radius, defined by the expression
\[\underline{\varrho}(\mathsf{A}):=\lim_{n \to \infty}\inf\left\{\|A_n \cdots A_1\|^{\frac{1}{n}} \colon A_i \in \mathsf{A}\right\}=\inf_{(A_i)_{i=1}^\infty \in \mathsf{A}^{\mathbb{N}}} \limsup_{n \to \infty}\|A_n \cdots A_1\|^{\frac{1}{n}}\]
has been studied relatively little, despite its relevance to a range of pure and applied mathematical problems (see e.g. \cite{BlCaJu09,DeKu11,DuSiTh99,Gu95,Mo16,Pr00,TsBl97}).

By comparison with the joint spectral radius, the behaviour of the lower spectral radius is delicate. The joint spectral radius of a finite set of matrices depends continuously on the matrix entries \cite{HeSt95} and there exists an algorithm to determine whether or not $\overline{\varrho}(\mathsf{A})=0$ using a number of steps which grows polynomially with the dimension of the matrices and with the cardinality of the set of matrices being considered (\cite{Gu96}, see also \cite[\S2.3.1]{Ju09}). Conversely the lower spectral radius can exhibit discontinuities (see \cite{BoMo13,Ju09}) and the problem of determining whether or not a pair of $15 \times 15$ integer matrices has lower spectral radius equal to zero is computationally undecidable \cite{CaHaHaNi14}.

A set of matrices $\sA$ is said to satisfy the \emph{finiteness property} for the joint spectral radius $\overline{\varrho}$ -- a situation which in this article will be called the \emph{upper finiteness property} --  if there exists a finite sequence $A_1,\ldots,A_n$ of elements of $\sA$, possibly including repetitions, such that $\overline{\varrho}(\sA)=\rho(A_1\cdots A_n)^{1/n}$, where $\rho$ denotes the ordinary spectral radius of a square matrix. The upper finiteness property has been investigated extensively since it was first defined explicitly in \cite{LaWa95} (and implicitly in \cite{Gu95}) and is believed to hold for ``most'' finite sets of matrices (though see e.g. \cite{BlThVl03,BoMa02,HaMoSiTh11} for examples where the property fails). By analogy we will say that $\sA$ has the \emph{lower finiteness property} if $\underline{\varrho}(\sA)=\rho(A_1 \cdots A_n)^{1/n}$ for some finite sequence $A_1,\ldots,A_n$ of elements of $\sA$. The most powerful methods presently available for computing the upper and lower spectral radii of a finite set of matrices proceed by proving that the set of matrices in question has the upper or lower finiteness property, and by exhibiting precisely a product whose normalised spectral radius $\rho(A_1 \cdots A_n)^{1/n}$  realises the upper or lower spectral radius \cite{GuPr13}. In the interest of the computation of these joint spectral characteristics it is therefore important to understand exactly how frequently these finiteness properties can be expected to hold.

In the article \cite{BoMo15}, J. Bochi and the present author investigated the continuity properties of the lower spectral radius both in the context of compact sets of matrices (with respect to the Hausdorff metric) and as a function of finite sets of matrices with fixed cardinality. As a corollary, results were also obtained on the genericity of the lower finiteness property. In order to state these results fully we require a definition:
\begin{definition}
Let $A_1, A_2$ be $2 \times 2$ real matrices. We will say that $(A_1,A_2)$ is a \emph{resistant pair} if there exist $c,\varepsilon,\gamma>0$ with the following property: for all $n \geq 1$ and for all choices of integers $i_1,\ldots,i_n \in \{1,2\}$, if at most $\varepsilon n$ of the integers $i_k$ are equal to $2$, then $\|A_{i_1}\cdots A_{i_n}\| \geq ce^{\gamma n}$.
\end{definition}
The following conjecture was formulated in \cite{BoFa06}:
\begin{conjecture}[Bochi-Fayad Conjecture]\label{co:bf}
Let $\mathcal{H}$ denote the set of all $2 \times 2$ real matrices with unit determinant and unequal real eigenvalues, and let $\mathcal{R}$ denote the set of all $2 \times 2$ real matrices unit determinant and non-real eigenvalues. Then the set of all resistant pairs $(H,R) \in \mathcal{H}\times\mathcal{R}$ has full Lebesgue measure.
\end{conjecture}
Some partial results in the direction of this conjecture may be found in \cite{AlSeUn11,AvRo09,FaKr08}; we note that resistant pairs $(H,R) \in \mathcal{H}\times \mathcal{R}$ can easily be constructed explicitly in the special case where $R^n$ is equal to the identity for some integer $n \geq 1$, see \cite{BoMo15}. Let $M_2(\mathbb{R})$ denote the set of all $2 \times 2$ real matrices. J. Bochi and the author have proved the following result:
\begin{theorem}[\cite{BoMo15}]\label{th:bm}
Define
\[\mathcal{U}:=\left\{(\alpha H, \beta R)\colon H \in \mathcal{H}, R \in \mathcal{R}\text{ and }0<\alpha<\beta\right\}\subset M_2(\mathbb{R})^2.\]
Then there is a dense $G_\delta$ subset of $\mathcal{U}$ in which the lower finiteness property fails. If additionally the Bochi-Fayad conjecture is true, then the set of discontinuities of the lower spectral radius on $\mathcal{U}$ has full Lebesgue measure.
\end{theorem}
Since $\mathcal{U}$ is open, this result motivates in particular the following conjecture:
\begin{conjecture}\label{co:cks}
The set of discontinuities of $\underline{\varrho}\colon M_2(\mathbb{R})^2\to\mathbb{R}$ has Hausdorff dimension $8$. 
\end{conjecture}

At the present time the resolution of Conjecture \ref{co:bf} seems remote. The purpose of this article is to investigate what partial results in the direction of Conjecture \ref{co:cks} may be obtained without the resolution of Conjecture \ref{co:bf}. In the process of this investigation we have also obtained results on the lower finiteness property.

Let us write $\mathcal{P}$ for the set of all $2\times2$ real matrices with rank one which are not nilpotent, and $\mathcal{E}$ for the set of all $2\times 2$ real matrices with non-real eigenvalues. In this article we investigate the size of the set of discontinuities of $\underline{\varrho}$ and the prevalence of the lower finiteness property on the $7$-dimensional manifold $\mathcal{P}\times\mathcal{E} \subset M_2(\mathbb{R})^2$, where the lower spectral radius is much more easily understood than on the set investigated in Theorem \ref{th:bm}. The situation which emerges is surprisingly subtle: we find that the set of discontininuities of $\underline{\varrho}$ on $\mathcal{P}\times\mathcal{E}$, and the set of pairs $(A_1,A_2)\in\mathcal{P}\times\mathcal{E}$ which satisfy the lower finiteness property, are both dense subsets of $\mathcal{P}\times\mathcal{E}$ with full Lebesgue measure; but on the other hand their complements are dense $G_\delta$ sets. The set of discontinuities in this set may therefore be seen as large in the sense of measure theory, but small in the sense of point-set topology, and similarly for the set in which the lower finiteness property holds. Yet furthermore, there exist dense subsets of $\mathcal{P}\times\mathcal{E}$ where $\underline{\varrho}$ is continuous but the lower finiteness property holds, and where $\underline{\varrho}$ is discontinuous but the lower finiteness property fails. We prove the following theorem:
\begin{theorem}\label{th:only}
Let $\mathcal{P} \subset M_2(\R)$ denote the set of all non-nilpotent rank-one matrices and $\mathcal{E}\subset M_2(\R)$ the set of all matrices with non-real eigenvalues, and define four subsets $\mathcal{U}_1,\mathcal{U}_2,\mathcal{U}_3,\mathcal{U}_4$ of $\mathcal{P}\times\mathcal{E}$ as follows:
\begin{itemize}
\item
$\mathcal{U}_1$ is the set of all $\mathsf{A} \in \mathcal{P}\times\mathcal{E}$ such that $\underline{\varrho}(\mathsf{A})=0$, $\underline{\varrho}$ is continuous at $\mathsf{A}$, and $\mathsf{A}$ \emph{does not} have the lower finiteness property.
\item
$\mathcal{U}_2$ is the set of all $\mathsf{A} \in \mathcal{P}\times\mathcal{E}$ such that $\underline{\varrho}(\mathsf{A})=0$, $\underline{\varrho}$ is continuous at $\mathsf{A}$, and $\mathsf{A}$ \emph{does} have the lower finiteness property.

\item
$\mathcal{U}_3$ is the set of all $\mathsf{A} \in \mathcal{P}\times\mathcal{E}$ such that $\underline{\varrho}(\mathsf{A})>0$, $\underline{\varrho}$ is discontinuous at $\mathsf{A}$, and $\mathsf{A}$ \emph{does not} have the lower finiteness property.

\item
$\mathcal{U}_4$ is the set of all $\mathsf{A} \in \mathcal{P}\times\mathcal{E}$ such that $\underline{\varrho}(\mathsf{A})>0$, $\underline{\varrho}$ is discontinuous at $\mathsf{A}$, and $\mathsf{A}$ \emph{does} have the lower finiteness property.
\end{itemize}
Then $\mathcal{P}\times\mathcal{E}$ is equal to the union of the four sets $\mathcal{U}_i$, and each of the four sets $\mathcal{U}_i$ is dense in $\mathcal{P}\times\mathcal{E}$. The set $\mathcal{U}_1$ is equal to a countable intersection of dense open subsets of $\mathcal{P}\times\mathcal{E}$: conversely, the set $\mathcal{U}_4$ is of full Lebesgue measure.
\end{theorem}
Since $\mathcal{P}\times\mathcal{E}$ is a seven-dimensional submanifold of the eight-dimensional linear space $M_2(\mathbb{R})^2$, we immediately obtain the following result in the direction of Conjecture \ref{co:cks}:
\begin{corollary}
The set of discontinuities of the lower spectral radius $\underline{\varrho}\colon M_2(\mathbb{R})^2\to\mathbb{R}$ has Hausdorff dimension at least $7$. 
\end{corollary}
The proof of Theorem \ref{th:only} is essentially divided into four parts: some common preliminaries, an analysis of the sets $\mathcal{U}_1$ and $\mathcal{U}_2$, an analysis of the set $\mathcal{U}_4$, and an analysis of the set $\mathcal{U}_3$. The last of these parts is by far the most intricate, and indeed is similar in length to the other three parts combined. The remainder of this paper follows the structure of this proof: in \S2 we prove certain preliminary statements; in \S3 we examine the case $\underline{\varrho}(\mathsf{A})=0$; in \S4 and \S5 we examine $\mathcal{U}_4$ and $\mathcal{U}_3$ respectively; and in \S6, all of these results are synthesised into the proof of Theorem \ref{th:only}.

\section{Preliminaries}

In this section we show that a general pair $(H,R) \in \mathcal{P}\times\mathcal{E}$ may be reduced by rescaling and a change of basis to the situation in which $R$ is a rotation and $H$ has only zero entries in its bottom row. Using this reduction we obtain a formula for the lower spectral radius which will be used in all subsequent sections. We begin with the following fundamental observation:
\begin{lemma}\label{le:similar}
Let $(H,R) \in \mathcal{P} \times \mathcal{E}$. Then there exists $\gamma>0$ such that $\gamma^{-1}(H,R)$ is simultaneously similar to a pair of matrices of the form
\[\left(\left(\begin{array}{cc}\lambda & \alpha \\0&0\end{array}\right),\left(\begin{array}{cc}\cos \theta & -\sin \theta\\\sin\theta&\cos \theta\end{array}\right)\right)\]
where $\lambda,\alpha,\theta \in \mathbb{R}$ and $\lambda \neq 0$.
\end{lemma}
\begin{proof}
Let $H,R \in \mathcal{P}\times \mathcal{E}$. By dividing both matrices by $(\det R)^{-1/2}$ if necessary we may assume that $\det R=1$; subject to this assumption we will show that the pair $(H,R)$ is similar to a pair of the desired form.

Since $R$ is a real matrix with unit determinant and complex eigenvalues, there exists an invertible matrix $A$ such that $ARA^{-1}$ is a rotation matrix. Let us write $H':=AHA^{-1}$. Since $H'$ has trace $\lambda \neq 0$ and determinant $0$, it has a one-dimensional eigenspace corresponding to the eigenvalue $\lambda$. It is clear that by conjugating $H'$ by a suitable rotation $B$ we may obtain a matrix $H'':=BH'B^{-1}$ for which the horizontal axis is an eigenspace corresponding to the eigenvalue $\lambda$; since rotations commute with one another, this additional conjugation leaves the rotation matrix $ARA^{-1}$ unchanged. Since $H''$ preserves the horizontal axis we may write
\[H''=\left(\begin{array}{cc}\lambda & \alpha \\0&\beta\end{array}\right)\]
for real numbers $\alpha$ and $\beta$, but since $\det H'' = \det H =0$ and $\lambda \neq 0$ we must have $\beta=0$. 
\end{proof}

\begin{lemma}\label{le:explicittrace}
Let $H \in \mathcal{P}$ and $R \in \mathcal{E}$, and suppose that we have
\begin{equation}\label{eq:simpleform}H=\left(\begin{array}{cc}\lambda & \alpha \\0&0\end{array}\right),\qquad R=\left(\begin{array}{cc}\cos \theta & -\sin \theta\\\sin\theta&\cos \theta\end{array}\right)\end{equation}
for some $\alpha,\lambda,\theta \in \R$, where $\lambda \neq 0$. Then for all finite sequences $n_1,\ldots,n_k$ and $m_1,\ldots,m_k$ of positive integers
\[\tr\left(H^{n_k}R_\theta^{m_k}\cdots H^{n_1}R_\theta^{m_1}\right)=\lambda^{\sum_{i=1}^k n_i} \prod_{i=1}^k \left(\cos m_i\theta + \alpha\lambda^{-1}\sin m_i\theta\right)\]
and hence in particular
\[\rho\left(H^{n_k}R_\theta^{m_k}\cdots H^{n_1}R_\theta^{m_1}\right)=|\lambda|^{\sum_{i=1}^k n_i} \prod_{i=1}^k \left|\cos m_i\theta + \alpha\lambda^{-1}\sin m_i\theta\right|.\]
\end{lemma}
\begin{proof}
By the Cayley-Hamilton theorem or direct computation we have $H^2=\lambda H$ and hence by a trivial induction $H^n=\lambda^{n-1}H$ for every $n \geq 1$. It follows that for each $n,m \geq 1$
\begin{align*}H^nR^m &= \left(\begin{array}{cc}\lambda^n&\lambda^{n-1}\alpha\\0&0\end{array}\right)\left(\begin{array}{cc}\cos m\theta &-\sin m\theta\\\sin m\theta & \cos m\theta\end{array}\right)\\
&= \left(\begin{array}{cc}\lambda^n(\cos m\theta+\alpha \lambda^{-1}\sin m\theta)&\lambda^n(\alpha\lambda^{-1}\cos m\theta - \sin m\theta)\\0&0\end{array}\right).\end{align*}
It follows directly that the lower-right and upper-left entries of the matrix
\[H^{n_k}R_\theta^{m_k}\cdots H^{n_1}R_\theta^{m_1}\]
are respectively zero and 
\[\lambda^{\sum_{i=1}^k n_i} \prod_{i=1}^k \left(\cos m_i\theta + \alpha\lambda^{-1}\sin m_i\theta\right)\]
and the result follows.
\end{proof}
\begin{lemma}\label{le:newformula}
Let $(H,R) \in \mathcal{P}\times\mathcal{E}$, and let $H^{n_k}R^{m_k}\cdots H^{n_1}R^{m_1}$ be a product of these two matrices in which every exponent $n_i$, $m_i$ is a non-negative integer, and not all of these integers are zero. Then there exists $n \geq 0$ such that
\begin{equation}\label{eq:mogisawesome}\rho\left(H^{n_k}R^{m_k}\cdots H^{n_1}R^{m_1}\right)^{\frac{1}{\sum_{i=1}^k (n_i+m_i)}}\geq \min\left\{\rho(R),\rho\left(HR^n\right)^{\frac{1}{n+1}}\right\}.\end{equation}
\end{lemma}
\begin{proof}
It is clear that the quantities appearing in the statement of the lemma are homogenous in $H$ and $R$ and are unaffected by a change-of-basis transformation, so by Lemma \ref{le:similar} it is sufficient to consider the case where $(H,R)$ has the form \eqref{eq:simpleform}. 

Let us assume that the product $H^{n_k}R^{m_k}\cdots H^{n_1}R^{m_1}$ has been presented in such a manner that the integer $k \geq 1$ is minimal. If all of the $m_i$ are zero then we have
\[\rho\left(H^{n_k}R^{m_k}\cdots H^{n_1}R^{m_1}\right)^{\frac{1}{\sum_{i=1}^k (n_i+m_i)}}=\rho\left(H^{\sum_{i=1}^kn_i}\right)^{\frac{1}{\sum_{i=1}^kn_i}}=\rho(H)\]
so that the conclusion of the lemma holds with $n:=0$. If all of the indices $n_i$ are zero then the left-hand side of \eqref{eq:mogisawesome} similarly reduces to $\rho(R)$ and the conclusion holds trivially. For the remainder of the proof we may assume without loss of generality that $\sum_{i=1}^km_i$ and $\sum_{i=1}^k n_i$ are nonzero.

If $n_i$ is equal to zero for some integer $i<k$ then we may write
\[\rho\left(H^{n_k}R^{m_k}\cdots H^{n_1}R^{m_1}\right)=\rho\left(H^{n_k}R^{m_k}\cdots H^{n_{i+1}}R^{m_{i+1}+m_i}H^{n_{i-1}} \cdots H^{n_1}R^{m_1}\right)\]
which requires only $k-1$ indices, contradicting the minimalitiy of $k$. Similarly if $n_k=0$ then
\[\rho\left(H^{n_k}R^{m_k}\cdots H^{n_1}R^{m_1}\right)=\rho\left(H^{n_{k-1}}R^{m_{k-1}}\cdots H^{n_1}R^{m_1+m_k}\right)\]
by the standard identity $\rho(AB)=\rho(BA)$, and the minimality of $k$ is again contradicted. Analogous considerations show that each $m_i$ must also be nonzero. We may therefore apply Lemma \ref{le:explicittrace} to obtain
\[\rho\left(H^{n_k}R^{m_k}\cdots H^{n_1}R^{m_1}\right)=|\lambda|^{\sum_{i=1}^k (n_i-1)} \prod_{i=1}^k \left|\lambda\cos m_i\theta + \alpha\sin m_i\theta\right|.\]
Let us define
\[r:=\min_{1 \leq i \leq k}\left|\lambda\cos m_i\theta + \alpha\sin m_i\theta\right|^{\frac{1}{1+m_i}}.\]
Obviously, there exists an integer $m \geq 1$ such that
\[r=\left|\lambda\cos m\theta + \alpha \sin m\theta\right|^{\frac{1}{1+m}}=\rho\left(HR^m\right)^{\frac{1}{1+m}}.\]
Now, it is clear that we may write
\begin{align*}\rho\left(H^{n_k}R^{m_k}\cdots H^{n_1}R^{m_1}\right)&=|\lambda|^{\sum_{i=1}^k (n_i-1)} \prod_{i=1}^k \left|\lambda\cos m_i\theta + \alpha\sin m_i\theta\right|\\
&\geq |\lambda|^{\sum_{i=1}^k(n_i-1)} r^{\sum_{i=1}^k (1+m_i)}.\end{align*}
If $r>|\lambda|=\rho(H)$ then we have 
\[\rho\left(H^{n_k}R^{m_k}\cdots H^{n_1}R^{m_1}\right)> |\lambda|^{\sum_{i=1}^k(n_i+m_i)}\geq\rho(H)^{\sum_{i=1}^k(n_i+m_i)}\]
and the conclusion of the lemma is is satisfied with $n:=0$. Otherwise we have
\[\rho\left(H^{n_k}R^{m_k}\cdots H^{n_1}R^{m_1}\right)\geq r^{\sum_{i=1}^k (n_i+m_i)}\]
and therefore
\[\rho\left(H^{n_k}R^{m_k}\cdots H^{n_1}R^{m_1}\right)^{\frac{1}{\sum_{i=1}^k (n_i+m_i)}}\geq r=\rho\left(HR^m\right)^{\frac{1}{1+m}}\]
as required. The proof is complete. 
\end{proof}
We immediately deduce the following:
\begin{lemma}\label{le:formula}
Let $(H,R) \in \mathcal{P}\times\mathcal{E}$. Then
\[\underline{\varrho}(H,R)=\inf_{n \geq 0}\rho\left(HR^n\right)^{\frac{1}{n+1}}.\]
\end{lemma}
\begin{proof}
We have
\[\rho(HR^n) \leq \|HR^n\|\leq \|H\|\cdot \|R^n\|\]
for every $n \geq 0$, and therefore
\[\inf_{n \geq 0}\rho\left(HR^n\right)^{\frac{1}{1+n}} \leq \limsup_{n \to \infty} \rho\left(HR^n\right)^{\frac{1}{1+n}} \leq \limsup_{n \to \infty} \left(\|H\|\cdot \|R^n\|\right)^{\frac{1}{1+n}}=\rho(R)\]
using Gelfand's formula. It follows by this inequality together with Lemma \ref{le:newformula} that
\[\underline{\varrho}(H,R)\geq \inf_{n \geq 0}\rho\left(HR^n\right)^{\frac{1}{n+1}},\]
and the general identity
\[\underline{\varrho}(A_1,\ldots,A_k)=\inf_{n \geq 1} \min_{1\leq i_1,\ldots,i_n \leq k} \rho\left(A_{i_1}\cdots A_{i_n}\right)^{\frac{1}{n}}\]
(see e.g. \cite[\S1.2.1]{Ju09}) yields the other direction of inequality.
\end{proof}
Finally, the following lemma will be used on two occasions in which it is helpful to estimate small quantities of the form $|\lambda\cos \theta+\alpha\sin \theta|$.
\begin{lemma}\label{le:ratio}
Let $\alpha, \lambda \in \mathbb{R}$ where $\lambda$ is nonzero. If $\alpha=0$ then define $\theta_0:=\pi/2$; otherwise, define $\theta_0:=\tan^{-1}(-\lambda/\alpha)\in\left(-\frac{1}{2},\frac{1}{2}\right)$, so that in either event we have $\lambda \cos\theta_0+\alpha \sin\theta_0=0$. Then there exists a constant $C_0>1$ such that for every $\theta \in \mathbb{R}$
\[\frac{1}{C_0}\mathrm{dist}(\theta, \theta_0+\pi\mathbb{Z}) \leq \left|\lambda \cos \theta +\alpha\sin \theta\right|\leq C_0\mathrm{dist}(\theta,\theta_0+\pi\mathbb{Z}),\]
where $\mathrm{dist}(\theta,a+b\mathbb{Z})$ denotes the distance from the real number $\theta$ to the nearest  real number of the form $a+b\ell$ with $\ell \in \mathbb{Z}$.
\end{lemma}
\begin{proof}
Define $f,g \colon \mathbb{R} \to \mathbb{R}$ by $f(\theta):=|\lambda \cos \theta + \alpha \sin\theta|$ and $g(\theta):=\mathrm{dist}(\theta,\theta_0+\pi\mathbb{Z})$. We must show that the ratio $f(\theta)/g(\theta)$ is uniformly bounded away from zero and infinity with respect to $\theta \in \mathbb{R} \setminus (\theta_0+\pi\mathbb{Z})$. Since both $f$ and $g$ are periodic with period $\pi$, it is sufficient to show that $f(\theta)/g(\theta)$ is bounded away from zero and infinity for $\theta$ belonging to the interval $\left[-\frac{\pi}{2},\frac{\pi}{2}\right]\setminus \{\theta_0\}$.  

The function $f/g \colon \left[-\frac{\pi}{2},\frac{\pi}{2}\right]\setminus \{\theta_0\} \to \mathbb{R}$ is clearly well-defined, continuous and nowhere zero, so to show that this function is bounded away from zero and infinity it is sufficient to show that $\lim_{\theta \to \theta_0} f(\theta)/g(\theta)$ exists and is nonzero. We have
\begin{align*}
\lim_{\theta \to \theta_0} \frac{f(\theta)}{g(\theta)}&=\lim_{h \to 0} \frac{f(\theta_0+h)}{g(\theta_0+h)}\\
&=\lim_{h \to 0} \frac{|\lambda \cos(\theta_0+h)+\alpha\sin(\theta_0+h)|}{|h|}\\
&=\lim_{h \to 0} \left|\frac{\lambda \cos(\theta_0+h)+\alpha \sin(\theta_0+h)-\lambda \cos(\theta_0)-\alpha \sin(\theta_0)}{h}\right|\\
&=\left|\alpha \cos \theta_0-\lambda \sin\theta_0\right|\end{align*}
so the result follows if we can show that $\alpha\cos\theta_0-\lambda\sin \theta_0$ is nonzero. If $\alpha\cos\theta_0-\lambda\sin\theta_0=0$ and $\cos\theta_0\neq 0$ then $\tan \theta_0=\alpha/\lambda$ and therefore $-\lambda/\alpha=\alpha/\lambda$, implying the equation $\alpha^2+\lambda^2=0$ which is impossible since both numbers are real and $\lambda\neq0$. If instead $\alpha\cos\theta_0-\lambda\sin\theta_0=0$  and also $\cos\theta_0=0$ then we have $\lambda=0$, a contradiction.
\end{proof}

\section{On the existence of zero products}

In this section we investigate those pairs $(H,R) \in\mathcal{P}\times\mathcal{E}$ such that there exists a finite product of the two matrices $H$ and $R$ which is equal to the zero matrix. We begin with the following result, which shows that this situation is, in a certain topological sense, uncommon.
\begin{lemma}\label{le:notmanyzeros}
For every two finite sequences of non-negative integers $n_1,\ldots,n_k$ and $m_1,\ldots,m_{k+1}$, with all integers except possibly $m_1$ and $m_{k+1}$ being nonzero, the set
\[Z:=\left\{(A,B) \in \mathcal{P}\times\mathcal{E} \colon B^{m_{k+1}}A^{n_k}B^{m_k} A^{n_{k-1}}\cdots A^{n_1}B^{m_1}=0\right\}\]
is closed and has empty interior.
\end{lemma}
\begin{proof}
Since every $B\in\mathcal{E}$ is invertible we note that the equation
\[B^{m_{k+1}}A^{n_k}B^{m_k} A^{n_{k-1}}\cdots A^{n_1}B^{m_1}=0\]
for some $(A,B) \in \mathcal{P} \times \mathcal{E}$ is equivalent to the equation
\[A^{n_k}B^{m_k} A^{n_{k-1}}\cdots A^{n_1}B^{m_1+\ell}=0\]
for the same $A$ and $B$ and every non-negative integer $\ell$, so we may assume without loss of generality that $m_{k+1}=0$ and $m_1>0$. By continuity the set $Z$ is obviously closed.  For each $A \in \mathcal{P}$ let us define a function $\Phi_A \colon M_2(\R)\to M_2(\R)$ by
\[\Phi_{A}(B):=A^{n_k} B^{m_k} \cdots A^{n_1}B^{m_1}.\]
 Let us suppose for a contradiction that the set $Z$ contains a nonempty open set, and let $(A_0,B_0)$ be an interior point of $Z$. 
Each of the four entries of the matrix $\Phi_{A_0}(B)$ is a polynomial function of the entries of $B$, and by hypothesis is identically zero in a neighbourhood of $B_0$. It follows that each of these entries is the zero polynomial, and therefore $\Phi_{A_0}$ is identically zero throughout $M_2(\R)$; however, we obviously have
\[\Phi(A_0)=A_0^{\sum_{i=1}^k n_i+m_i} \neq 0\]
since $A_0$ is not nilpotent, and we obtain a contradiction. We conclude that $Z$ has empty interior as claimed.
\end{proof}

In the converse direction, the existence of a product of the matrices $(H,R)$ which is zero is in fact a dense property in $\mathcal{P}\times\mathcal{E}$:
\begin{lemma}\label{le:dense-zeros}
The set of all $(H,R) \in \mathcal{P}\times\mathcal{E}$ such that $\underline{\varrho}(H,R)=0$ and $(H,R)$ satisfies the lower finiteness property is a dense set.
\end{lemma}
\begin{proof}
Given $(H,R)\in\mathcal{P}\times\mathcal{E}$ we must show that there exist matrices $\tilde{H},\tilde{R}$ arbitrarily close to $(H,R)$ such that $\underline{\varrho}(\tilde{H},\tilde{R})=0$ and such that $\tilde{H},\tilde{R}$ satisfies the lower finiteness property. Clearly this property is unaffected by applying a change of basis or multiplying both matrices $H$ and $R$ by the same scalar, so using Lemma \ref{le:similar}  we may assume without loss of generality that the pair $(H,R)$ has the form
\[H=\left(\begin{array}{cc}\lambda & \alpha \\0&0\end{array}\right),\qquad R=\left(\begin{array}{cc}\cos \theta & -\sin \theta\\\sin\theta&\cos \theta\end{array}\right)\]
for some real numbers $\alpha,\lambda,\theta \in \R$, where $\lambda \neq 0$. If $(H,R)$ is as above for some $\alpha,\lambda,\theta \in \R$ then using Lemma \ref{le:explicittrace} we have
\[\tr R^mH=\lambda\cos m\theta +\alpha \sin m\theta\]
and if this trace is zero then by the Cayley-Hamilton theorem we have $(R^mH)^2=0$, which implies $HR^m H=0$ since $R$ is invertible. To prove the lemma it is thus sufficient to prove the following: for every fixed $\alpha,\lambda\in \R$ the set
\[\bigcup_{m=1}^\infty \{\theta \in \R \colon \lambda \cos m\theta+\alpha \sin m\theta =0\}\]
is dense in $\R$. Given $\lambda$ and $\alpha$, it is clearly sufficient to show that for every $m \geq 1$, every interval in $\R$ with length greater than $\pi/m$ contains a number $\theta$ such that $\lambda\cos m\theta + \alpha \sin m\theta=0$. But this result is straightforward to prove: if we define $\xi_m \colon \R \to \R$ by $\xi_m(\phi):=|\lambda\cos m\phi +\alpha \sin m\phi|$ then this function is periodic with period $\pi/m$, and it has at least one zero, since either $\xi_m(\pi/2m)=0$ if $\alpha=0$, or $\xi_m(m^{-1}\tan^{-1}(-\lambda/\alpha))=0$ if $\alpha \neq 0$. In particular every interval of length at least $\pi/m$ contains a zero of $\xi_m$ as required.
\end{proof}

Finally, we note the following implication of the existence of a zero product:

\begin{lemma}\label{le:zero-cont}
If $H,R \in M_2(\mathbb{R})$ and $\underline{\varrho}(H,R)=0$ then $\underline{\varrho} \colon M_2(\mathbb{R})^2 \to \mathbb{R}$ is continuous at $(H,R)$. 
\end{lemma}
\begin{proof}
Since for any $A_1,A_2 \in M_2(\mathbb{R})$ we have
\[\underline{\varrho}(A_1,A_2)=\inf_{n \geq 1} \inf_{i_1,\ldots,i_n \in \{1,2\}} \left\|A_{i_1}\ldots A_{i_n}\right\|^{\frac{1}{n}}\]
(see e.g. \cite[\S1.2.2]{Ju09}) the lower spectral radius of $(A_1,A_2)$ is equal to the infimum of a sequence of continuous functions of $(A_1,A_2)$ and hence is upper semi-continuous. Since it is also non-negative, this implies that every zero of the lower spectral radius is a point of continuity.
\end{proof}

\section{Behaviour on a set of full Lebesgue measure}
In this section we shall prove that the set of all $(H,R) \in \mathcal{P}\times\mathcal{E}$ such that the lower finiteness property is satisfied and such that the lower spectral radius of $(H,R)$ is nonzero is a set of full Lebesgue measure. 
In this section we shall find it convenient to consider angles $\theta$ as belonging to the quotient space $\R/2\pi\Z$ -- on which Lebesgue measure is a finite measure -- as opposed to $\R$, which has infinite Lebesgue measure.

Our first objective shall be to prove the following proposition:
\begin{proposition}\label{pr:measure-one}
Fix $\lambda,\alpha \in \R$ with $\lambda \neq 0$ and define
\[H:=\left(\begin{array}{cc}\lambda & \alpha \\0&0\end{array}\right),\qquad R_\theta:=\left(\begin{array}{cc}\cos \theta & -\sin \theta\\\sin\theta&\cos \theta\end{array}\right)\]
for every $\theta \in \R/2\pi\Z$. Then for Lebesgue almost every $\theta \in \R/2\pi\Z$ there exists $n \geq 0$ such that
\[\underline{\varrho}(H,R_\theta)=\rho(HR_\theta^n)^{\frac{1}{n+1}}>0.\]
\end{proposition}
Let us make some observations on the proof of Proposition \ref{pr:measure-one}. By Lemma \ref{le:formula} we find that
\[\underline{\varrho}(H,R_\theta)=\inf_{n \geq 0} \rho(HR_\theta^n)^{\frac{1}{n+1}}\]
for every $\theta \in \R/2\pi\Z$, and by Lemma \ref{le:explicittrace} we have
\[\rho(HR_\theta^m)^{\frac{1}{m+1}}=\left|\lambda \cos m\theta + \alpha \sin m\theta\right|^{\frac{1}{m+1}}\]
for every $\theta \in \R/2\pi\Z$ and $m \geq 0$. To prove the proposition it is therefore sufficient to show that for Lebegue almost all $\theta \in \R/2\pi\Z$ the infimum
\[\inf_{m \geq 0}\left|\lambda \cos m\theta + \alpha \sin m\theta\right|^{\frac{1}{m+1}}\]
is attained for some integer $m \geq 0$, and is nonzero. The proposition thus follows from the combination of Lemmas \ref{le:ae-lemma-1}, \ref{le:ae-lemma-2} and \ref{le:ae-lemma-3} below:
\begin{lemma}\label{le:ae-lemma-1}
Let $\lambda, \alpha \in \R$. If $\theta \in \R$ is not a rational multiple of $\pi$, then
\begin{equation}\label{eq:moragisawesome}\inf_{ m \geq 0}\left|\lambda \cos m\theta + \alpha \sin m\theta\right|^{\frac{1}{m+1}}<1.\end{equation}
In particular \eqref{eq:moragisawesome} holds for Lebesgue almost every $\theta \in \R/2\pi\Z$.
\end{lemma}
\begin{proof}
For each integer $m$ the value of $|\lambda\cos m\theta + \alpha \sin m\theta|$ is unaffected by adding an integer multiple of $\pi$ to $\theta$, so we may freely identify $\theta \in \R$ with its equivalence class modulo $2\pi\Z$. To prove \eqref{eq:moragisawesome} it is clearly sufficient to show that
\begin{equation}\label{eq:inf-1}\inf\left\{ |\lambda\cos m\theta + \alpha \sin m\theta| \colon m \geq 0\right\}<1.\end{equation}
When $\theta$ is not a rational multiple of $\pi$ the set $\{m\theta \in \mathbb{R}/2\pi\mathbb{Z} \colon m \geq 0\}$ is dense in $\mathbb{R}/2\pi\mathbb{Z}$, so by continuity we have
\begin{equation}\label{eq:inf-2}\inf\left\{  |\lambda\cos m\theta + \alpha\sin m\theta| \colon m \geq 0\right\}=\min\left\{  |\lambda\cos \phi + \alpha \sin \phi| \colon \phi \in \mathbb{R}/2\pi\mathbb{Z}\right\}.\end{equation}
If $\alpha=0$ then $|\lambda\cos (\pi/2) + \alpha\sin (\pi/2)|=0$, and if $\alpha \neq 0$ then we have $|\lambda\cos \phi_0+\alpha\sin\phi_0|=0$ where $\phi_0 \in \mathbb{R}/2\pi\mathbb{Z}$ satisfies  $\tan \phi_0=-\lambda /\alpha$. In either case the minimum in \eqref{eq:inf-2} is zero and hence the inequality \eqref{eq:inf-1} is satisfied. Since for Lebesgue almost every $\theta \in \mathbb{R}/2\pi\mathbb{Z}$, $\theta$ is not a rational multiple of $\pi$, the result follows.
\end{proof}
\begin{lemma}\label{le:ae-lemma-2}
Let $\lambda,\alpha \in \mathbb{R}$. Then for Lebesgue almost every $\theta \in \mathbb{R}/2\pi\mathbb{Z}$,
\[\lim_{m \to \infty} \left|\lambda\cos m\theta + \alpha \sin m\theta\right|^{\frac{1}{m+1}}=1.\]
\end{lemma}
\begin{proof}
Since trivially  $ |\lambda\cos m\theta + \alpha \sin m\theta| \leq |\lambda|+|\alpha|$ for every $m \geq 0$  we have
\[\limsup_{m \to \infty} \left|\lambda\cos m\theta + \alpha \sin m\theta\right|^{\frac{1}{m+1}} \leq 1\]
for every $\theta \in \mathbb{R}/2\pi\mathbb{Z}$. To prove the lemma we must show that the  limit inferior is greater than or equal to one for Lebesgue almost every $\theta \in \mathbb{R}/2\pi\mathbb{Z}$.

By Lemma \ref{le:ratio} there exist constants $C_0>1$ and $\theta_0 \in (- \frac{\pi}{2},\frac{\pi}{2})$ such that for every $\theta \in \R$ we have
\[\min_{k \in \mathbb{Z}}|\theta-\theta_0-k\pi|\leq C_0|\lambda\cos\theta+\alpha\sin\theta|.\]
In particular if $0 \leq \theta<2\pi$ and $|\lambda\cos\theta+\alpha\sin\theta|<t<\pi/2C_0$, then $\theta$ must belong to the set
\[(\theta_0-C_0t,\theta_0+C_0t)\cup (\theta_0+\pi-C_0t,\theta_0+\pi+C_0t)\cup (\theta_0+2\pi-C_0t,\theta_0+2\pi+C_0t).\]
We deduce that for every $t\in(0,\frac{\pi}{2C_0})$ the set
\[\left\{\theta \in [0,2\pi) \colon |\lambda\cos\theta+\alpha\sin\theta|<t\right\}\]
has Lebesgue measure not greater than $6C_0t$, and clearly this implies 
\[\mu\left(\left\{\theta \in \mathbb{R}/2\pi\mathbb{Z} \colon |\lambda\cos\theta+\alpha\sin\theta|<t\right\}\right)\leq 6C_0t\]
for the same range of real numbers $t$, where $\mu$ denotes Lebesgue measure on $\R/2\pi\Z$.
Since for every $m \geq 1$ the transformation $\mathbb{R}/2\pi\mathbb{Z}\to \mathbb{R}/2\pi\mathbb{Z}$ defined by $\theta \mapsto m\theta$ preserves Lebesgue measure, it follows that for all $m \geq 1$ and $t\in(0,\frac{\pi}{2C_0})$
\[\mu\left(\left\{\theta \in \mathbb{R}/2\pi\mathbb{Z} \colon |\lambda\cos m\theta+\alpha\sin m\theta|<t\right\}\right) \leq 6C_0t.\]
In particular
\[\mu\left(\left\{\theta \in \mathbb{R}/2\pi\mathbb{Z} \colon |\lambda\cos m\theta + \alpha \sin m\theta|<e^{-\sqrt{m}}\right\}\right)\leq 6C_0e^{-\sqrt{m}}\]
for every sufficiently large integer $m$. By the Borel-Cantelli Lemma we deduce that for Lebesgue almost every $\theta \in \mathbb{R}/2\pi\mathbb{Z}$ we have $|\lambda\cos m\theta + \alpha \sin m\theta|\geq e^{-\sqrt{m}}$ for every sufficiently large $m$, and therefore
\[\liminf_{m \to \infty} \left|\lambda\cos m\theta + \alpha\sin m\theta\right|^{\frac{1}{m+1}} \geq \liminf_{m \to \infty} \left(e^{-\sqrt{m}}\right)^{\frac{1}{m+1}}=1\]
as required to complete the proof.
\end{proof}
\begin{lemma}\label{le:ae-lemma-3}
Let $\alpha \in \mathbb{R}$. Then for  Lebesgue almost every $\theta \in \mathbb{R}/2\pi\mathbb{Z}$
\[\inf_{m \geq 0} \left|\lambda\cos m\theta + \alpha \sin m\theta\right|^{\frac{1}{m+1}}>0.\]
\end{lemma}
\begin{proof}
It follows from the combination of Lemma \ref{le:ae-lemma-1} and Lemma \ref{le:ae-lemma-2} that for Lebesgue almost every $\theta \in \mathbb{R}/2\pi\mathbb{Z}$ the above infimum is attained at some integer $n\geq 0$ depending on $\theta$. To prove the lemma it is therefore sufficient to show that the set
\[\left\{\theta \in \mathbb{R}/2\pi\mathbb{Z} \colon \exists\text{ }m \geq 0\text{ such that }\lambda\cos m\theta + \alpha \sin m\theta=0\right\}\]
has measure zero; but this set is countable, since for each $m \geq 0$ the set
\[\left\{\theta \in \mathbb{R}/2\pi\mathbb{Z} \colon \lambda\cos m\theta + \alpha \sin m\theta=0\right\}\]
is finite, and therefore it has Lebesgue measure zero as required.
\end{proof}

As was remarked previously, the combination of Lemmas \ref{le:ae-lemma-1}, \ref{le:ae-lemma-2} and \ref{le:ae-lemma-3} proves Proposition \ref{pr:measure-one}. We may now deduce the following result, which is the main result of this section:
\begin{proposition}\label{pr:full-measure-2}
The set
\[X:=\left\{(H,R)\in \mathcal{P}\times\mathcal{E} \colon \underline{\varrho}(H,R)=\rho(HR^n)^{\frac{1}{n+1}}>0\text{ for some }n \geq 0\right\}\]
has full Lebesgue measure as a subset of $\mathcal{P}\times\mathcal{E}$.
\end{proposition}
\begin{proof}
Let us define $\Lambda \subseteq SL_2^\pm(\R)\times (\R \setminus \{0\})^2 \times \R \times \mathbb{R}/2\pi\mathbb{Z}$ to be the set of all $(A,\gamma,\lambda,\alpha,\theta)$ such that the pair
\[\left(\left(\begin{array}{cc}\lambda & \alpha \\0&0\end{array}\right),\left(\begin{array}{cc}\cos \theta & -\sin \theta\\\sin\theta&\cos \theta\end{array}\right)\right)\]
does not belong to $X$. By Proposition \ref{pr:measure-one} this set has zero Lebesgue measure as a subset of $SL_2^\pm(\R)\times (\R \setminus \{0\})^2 \times \R \times \mathbb{R}/2\pi\mathbb{Z}$. Let us now define a function $\Phi \colon SL_2^\pm(\R)\times (\R \setminus \{0\})^2 \times \R \times \mathbb{R}/2\pi\mathbb{Z}\to \mathcal{P}\times\mathcal{E}$ by
\[\Phi(A,\gamma,\lambda,\alpha,\theta) =\gamma\left(A\left(\begin{array}{cc}\lambda & \alpha \\0&0\end{array}\right)A^{-1},A \left(\begin{array}{cc}\cos \theta & -\sin \theta\\\sin\theta&\cos \theta\end{array}\right)A^{-1}\right).\]
Using Lemma \ref{le:similar} it is easy to see that $(H,R) \in \mathcal{P}\times\mathcal{E}$ fails to belong to $X$ if and only if $(H,R) \in \Phi(\Lambda)$. The function $\Phi$ is a smooth map between 7-dimensional manifolds, so it is locally Lipschitz and maps sets whose  7-dimensional Lebesgue measure is equal to zero onto other sets whose 7-dimensional Lebesgue measure is equal to zero. In particular $\Phi(\Lambda)$ has Lebesgue measure zero as required.
\end{proof}

\section{Positive lower spectral radius without the lower finiteness property}
In this section we investigate those pairs $(H,R) \in \mathcal{P}\times\mathcal{E}$ for which the lower spectral radius is positive but the lower finiteness property is not satisfied, showing that the set of all such pairs in dense in $\mathcal{P}\times\mathcal{E}$. This outcome is obtained as a corollary of the following result:
\begin{proposition}\label{pr:ctdfrn}
Fix $\lambda \in \mathbb{R}\setminus \{0\}$. For each $\alpha \in \mathbb{R}$ and each $\theta \in \mathbb{R}/2\pi\mathbb{Z}$ define
\[H_\alpha:=\left(\begin{array}{cc}\lambda & \alpha \\0&0\end{array}\right),\qquad R_\theta:=\left(\begin{array}{cc}\cos \theta & -\sin \theta\\\sin\theta&\cos \theta\end{array}\right).\]
Then the set of all $(\alpha,\theta) \in \R^2$ such that the pair $(H_\alpha,R_\theta)$ has positive lower spectral radius and does not have the lower finiteness property is a dense subset of $\R^2$.
\end{proposition}

If $\alpha, \lambda \in \R$ are fixed, with $\lambda \neq 0$, then Lemma \ref{le:formula} together with Lemma \ref{le:explicittrace} tells us that we have
\[\underline{\varrho}(H_\alpha,R_\theta)=\inf_{m \geq 0}\rho(H_\alpha R_\theta^m)^{\frac{1}{1+m}}=\inf_{m \geq 0}|\lambda \cos m\theta +\alpha \sin m\theta|^{\frac{1}{1+m}}.\]
On the other hand the combination of Lemma \ref{le:newformula} and Lemma \ref{le:explicittrace} tells us that the lower finiteness property for $(H_\alpha,R_\theta)$ holds if and only if either $\underline{\varrho}(H_\alpha,R_\theta)=\rho(R_\theta)=1$, or there exists $m \geq 0$ such that
\[\underline{\varrho}(H_\alpha,R_\theta)=|\lambda \cos m\theta +\alpha \sin m\theta|^{\frac{1}{1+m}}.\]
Now, Lemma \ref{le:ratio} tells us that the quantity $|\lambda \cos m\theta +\alpha \sin m\theta|$ is well-approximated by the quantity $\mathrm{dist}(m\theta,\theta_0+\pi\Z)$, where $\theta_0$ is the unique zero of the function $\phi \mapsto |\lambda \cos \phi +\alpha \sin \phi|$ in $(-\frac{\pi}{2},\frac{\pi}{2}]$. The behaviour of $|\lambda \cos m\theta +\alpha \sin m\theta|^{1/(1+m)}$ for large $m$ is thus linked with the following problem of inhomogenous Diophantine approximation: for large $m$, how close to an integer is $\pi^{-1}(m\theta-\theta_0)$?

To investigate this problem we apply the theory of continued fractions. In the sequel we will refer to the classic treatment of the subject by A. Ya. Khinchin for specific results. We briefly recall some relevant definitions. Recall that if $a_1,a_2,\ldots$ is a finite or infinite sequence of positive integers then we define the sequence of \emph{convergents} $p_k/q_k$ by
\[\frac{p_0}{q_0}:=\frac{0}{1},\qquad\frac{p_1}{q_1}:=\frac{1}{a_1},\qquad \frac{p_2}{q_2}:=\frac{1}{a_1+\frac{1}{a_2}},\qquad\frac{p_3}{q_3}:=\frac{1}{a_1+\frac{1}{a_2+\frac{1}{a_3}}},\]
\[\frac{p_n}{q_n}:=\frac{1}{a_1+\frac{1}{a_2+\frac{1}{a_3+\cdots \frac{1}{a_n}}}}\]
for each $n$ for which the integers $a_1,\ldots,a_n$ are defined. The above representations $p_k/q_k$ in least terms can equivalently be defined by the recurrence relations
\begin{equation}\label{eq:stdcvgt}p_{k+1}:=a_{k+1}p_k+p_{k-1},\qquad q_{k+1}:=a_{k+1}q_k+q_{k-1}\end{equation}
for every $k \geq 0$, where we additionally define $p_{-1}:=1, q_{-1}:=0$.
If $\theta \in (0,1)$ is irrational then there exists a unique sequence of positive integers $(a_n)_{n=1}^\infty$ for which the associated sequence of convergents $p_k/q_k$ converges to $\theta$, and for every sequence of positive integers $(a_n)_{n=1}^\infty$ the corresponding sequence of convergents is a convergent sequence with an irrational limit. If $(a_n)$ is such a sequence then we denote the limit of the corresponding sequence of convergents by $[a_1,a_2,\ldots]$; more generally, if $(a_n)_{n=0}^\infty$ is a sequence such that $a_0 \in \mathbb{Z}$ and $a_n \in \mathbb{N}$ for all $n \geq 1$, then we write $[a_0;a_1,a_2,\ldots]:=a_0+[a_1,a_2,\ldots]$. The signifance of continued fractions for this work is their property of being good rational approximations: if $\theta=[a_1,a_2,\ldots]$ then we have
\[\frac{1}{q_{n+1}+q_n}<\left|q_n\theta-p_n\right|<\frac{1}{q_{n+1}}\]
for every $n \geq 1$, and if $q_n<\ell<q_{n+1}$ then
\[\left|\ell\theta-k\right|\geq \frac{1}{2\ell}\]
for every $k \in \mathbb{Z}$ which shares no common factors with $\ell$, see for example Theorems 9, 13 and 19 of \cite{Kh97}. The proof of Proposition \ref{pr:ctdfrn} will be executed by constructing, for each fixed $\lambda$ and in the case where $\alpha \in \R$ belongs to a suitable dense set, a dense set of $\theta \in \mathbb{R}$ whose convergents satisfy a certain precise growth condition.

The proof of Proposition \ref{pr:ctdfrn} begins with the following preliminary result:
\begin{lemma}\label{le:minor-lemma-ctd-frn}
Let $\theta_0 \in (-\frac{b}{2},\frac{b}{2})\setminus \mathbb{Q}$ and $\frac{a}{b} \in \mathbb{Q}$ in least terms, where $b$ is prime and $a$ is nonzero, and let $\varepsilon>0$. Then there exist integers $N \geq 1$ and $a_0,a_1,\ldots,a_N$ with the following properties: the associated sequence of convergents $(p_k/q_k)_{k=0}^N$ satisfies
\begin{equation}\label{eq:galo}p_N+a_0q_N \equiv p_{N-1}+a_0q_{N-1} \equiv a \mod b,\end{equation}
and if $(a_n)_{n=0}^\infty$ is any infinite extension of the finite sequence $a_0,a_1,\ldots,a_N$ then $|[a_0;a_1,a_2,\ldots]-\theta_0|<\varepsilon$.
\end{lemma}
\begin{proof}
Let $(b_n)_{n=0}^\infty$ be the unique sequence of integers such that $\theta_0=[b_0;b_1,b_2,\ldots]$, where $b_n \geq 1$ for every $n \geq 1$. Let $(P_n/Q_n)_{n=0}^\infty$ be the associated sequence of convergents, and let $n_0\geq 1$ be large enough that $1/Q_{n_0}^2<\varepsilon/2$. Define $a_k:=b_k$ for $0 \leq k \leq n_0$. We will show that there exists a finite extension $(a_n)_{n=0}^N$ of $(a_n)_{n=0}^{n_0}$ such that the associated sequence of convergents $(p_k/q_k)_{k=0}^N$ satisfies \eqref{eq:galo}. We will subsequently deduce that the second desired property also holds. 

Let us first show that it is impossible for the integers $p_{n_0}+a_0q_{n_0}$ and $p_{n_0-1}+a_0q_{n_0-1}$ to both be divisible by $b$. We prove this claim by contradiction. If $p_k+a_0q_k$ and $p_{k-1}+a_0q_{k-1}$ are both divisible by $b$ for some integer $k \geq 1$, then it follows from the relation \eqref{eq:stdcvgt} that
\[p_k+a_0q_k = a_k(p_{k-1}+a_0q_{k-1})+p_{k-2}+a_0q_{k-2}\]
and therefore $p_{k-2}+a_0q_{k-2}$ is also divisible by $b$. If $p_{n_0}+a_0q_{n_0}$ and $p_{n_0-1}+a_0q_{n_0-1}$ are both divisible by $b$, then by inductive descent it follows that $p_{-1}+a_0q_{-1}$ is divisible by $b$, but $p_{-1}+a_0q_{-1}=1$, which is a contradiction since $b$ is prime. This contradiction proves the claim.

Let us consider the possibilities which remain. If  $p_{n_0}+a_0q_{n_0}$ is not divisible by $b$, then since $b$ is prime we may choose an integer $k\geq 1$ such that
\[k(p_{n_0}+a_0q_{n_0})+p_{n_0-1}+a_0q_{n_0-1} \equiv a \mod b.\]
Taking $a_{n_0+1}:=k$ we find that $p_{n_0+1}+a_0q_{n_0+1} \equiv a \mod b$. If we now take $a_{n_0+2}:=b$ then 
\[p_{n_0+2}+a_0q_{n_0+2} = b(p_{n_0+1}+a_0q_{n_0+1})+p_{n_0}+a_0q_{n_0} \equiv a \mod b\]
so we can now take $N:=n_0+2$ and we have obtained the desired extension of $(a_n)_{n=0}^{n_0}$ in this case. If on the other hand $p_{n_0}+a_0q_{n_0}$ \emph{is} divisible by $b$, then by the previous claim $p_{n_0-1}+a_0q_{n_0-1}$ is not. Taking $a_{n_0+1}:=1$ we have
\[p_{n_0+1}+a_0q_{n_0+1} =p_{n_0}+a_0q_{n_0}+p_{n_0-1}+a_0q_{n_0-1}\equiv p_{n_0-1}+a_0q_{n_0-1} \mod b.\]
This reduces the problem to the preceding case but with $n_0+1$ in place of $n_0$. In either event we are able to construct a finite extension $(a_n)_{n=0}^N$ of $(a_n)_{n=0}^{n_0}$ such that the associated sequence of convergents $(p_k/q_k)_{k=0}^N$ satisfies \eqref{eq:galo} as desired.

Let us now suppose that $(a_n)_{n=0}^\infty$ is an infinite extension of the finite sequence $(a_n)_{n=0}^N$ with $a_n \geq 1$ for every $n >N$, and let $(p_k/q_k)_{k=0}^\infty$ be the associated sequence of convergents. Let $\theta:=[a_0;a_1,a_2,\ldots] \in \R$. The sequence of convergents satisfies $p_k/q_k=P_k/Q_k$ for $k \leq n_0$ since the sequences $(a_n)$ and $(b_n)$ coincide up to this point. By \cite[Theorem 16]{Kh97} it follows that
\[\left|\theta-\theta_0\right| \leq \left|\theta-\frac{P_{n_0}}{Q_{n_0}}\right|+\left|\frac{P_{n_0}}{Q_{n_0}}-\theta_0\right| \leq \frac{1}{Q_{n_0}^2}+\frac{1}{Q_{n_0}^2} <\varepsilon\]
since by the definition of $n_0$ we have $1/Q_{n_0}^2<\varepsilon/2$. The proof is complete.
\end{proof}

The core of the proof of Proposition \ref{pr:ctdfrn} rests in the following long lemma:
\begin{lemma}\label{le:main-lemma-ctd-frn}
Let $\frac{a}{b} \in \mathbb{Q}$ in least terms where $b$ is prime and $a$ is nonzero, let $K>1$, let $\theta_0 \in (-\frac{b}{2},\frac{b}{2})\setminus \mathbb{Q}$ and let $\varepsilon>0$. Then there exists an infinite sequence of integers $(a_n)_{n=0}^\infty$, where $a_n \geq 1$ for every $n \geq 1$, such that the irrational number $\theta:=[a_0;a_1,a_2,\ldots]$ belongs to $(-\frac{b}{2},\frac{b}{2})$, satisfies $|\theta-\theta_0|<\varepsilon$, and has the following additional properties: for every $n \geq 0$ there exists $q>n$ such that
\begin{equation}\label{eq:first-lemma-bit}\mathrm{dist}(n\theta,a+b\mathbb{Z})^{\frac{1}{n+1}}>K^{\frac{1}{n+1}}\mathrm{dist}(q\theta,a+b\mathbb{Z})^{\frac{1}{q+1}},\end{equation}
and furthermore
\begin{equation}\label{eq:second-lemma-bit}\inf_{n \geq 0}\mathrm{dist}(n\theta,a+b\mathbb{Z})^{\frac{1}{n+1}}>0.\end{equation}
\end{lemma}
\begin{proof}
By replacing $\varepsilon>0$ with a smaller value if necessary, we may ensure that $|\theta-\theta_0|<\varepsilon$ implies that $\theta \in (-\frac{b}{2},\frac{b}{2})$. By Lemma \ref{le:minor-lemma-ctd-frn} there exist integers $N \geq 1$ and $a_0 \in \Z$, $a_1,\ldots,a_N \in \mathbb{N}$ such that every infinite extension $(a_n)_{n=0}^\infty$ of the finite sequence $(a_n)_{n=0}^N$ has the property that $|[a_0;a_1,\ldots]-\theta_0|<\varepsilon$, and such that the associated sequence of convergents $(p_k/q_k)_{k=0}^N$ satisfies
\[p_N+a_0q_N \equiv p_{N-1}+a_0q_{N-1}\equiv a \mod b.\]
We will show that this sequence $(a_n)_{n=0}^N$ may be extended to an infinite sequence $(a_n)_{n=0}^\infty$ in such a way that the associated number $\theta:=[a_0;a_1,\ldots]$ satisfies properties \eqref{eq:first-lemma-bit} and \eqref{eq:second-lemma-bit}.

Let $a_{N+1}$ be an integer which is divisible by $b$ and is so large that the integer $q_{N+1}:=a_{N+1}q_N+q_{N-1}$ satisfies
\begin{equation}\label{eq:bewgry}q_{N+1}^{\frac{1}{1+q_N}}> 2Kq_N>(2Kq_N)^{\frac{1}{1+q_{N-1}}}.\end{equation}
 Since $a_{N+1}$ is a multiple of $b$, and by hypothesis $p_N+a_0q_N \equiv p_{N-1}+a_0q_{N-1}\equiv a \mod b$, the equation
 \[p_{N+1}+a_0q_{N+1}=a_{N+1}(p_N+a_0q_N)+p_{N-1}+a_0	q_{N-1}\]
 implies
 \[p_{N+1}+a_0q_{N+1} \equiv p_N+a_0q_N \equiv a \mod b.\]
 Now let $C>1$ be so large that
 \[q_{N+1} \leq C^{1+q_N}.\]
 We claim that there exists an infinite sequence $(a_n)_{n=0}^\infty$ which extends the finite sequence $a_0,a_1,\ldots,a_{N+1}$ such that for each $n \geq N+1$,
 \begin{equation}\label{eq:induct-1}p_{n}+a_0q_{n} \equiv p_{n-1}+a_0q_{n-1}\equiv a \mod b\end{equation}
and
\begin{equation}\label{eq:induct-2}(2Kq_{n-1})^{\frac{1}{1+q_{n-2}}}<q_n^{\frac{1}{1+q_{n-1}}} \leq C(4Kb)^{\sum_{k=N}^{n-2}\frac{1}{1+q_k}}.\end{equation}
The sequence is constructed inductively. Let us assume that the sequence $(a_n)_{n=0}^m$ has been constructed in such a way that the above properties are satisfied for every $n$ in the range $N+1 \leq n \leq m$. (Clearly this already holds in the particular case $m=N+1$.) We will show that an integer $a_{m+1}$ may be chosen such that the extended sequence $(a_n)_{n=0}^{m+1}$ has the above properties for all integers $n$ in the range $N+1\leq n\leq m+1$.

 Given the sequence $(a_n)_{n=0}^m$, we define the integer $a_{m+1}$ by 
\[a_{m+1}:=b\left\lceil q_{m}^{-1}(2Kq_{m})^{\frac{1+q_{m}}{1+q_{m-1}}}\right\rceil.\]
Since
\[p_{m+1}+a_0q_{m+1}=a_{m+1}(p_m +a_0q_m)+p_{m-1}+a_0q_{m-1}\]
 and $b$ divides $a_{m+1}$, we have
 \[p_{m+1}+a_0q_{m+1} \equiv p_{m-1}+a_0q_{m-1} \equiv a \mod b,\]
 and we know already that $p_m +a_0q_m\equiv a \mod b$, so the property \eqref{eq:induct-1} is also satisfied for $n:=m+1$. For the integer $q_{m+1}:=a_{m+1}q_m+q_{m-1}$ we have the inequality
\begin{align*}q_{m+1}^{\frac{1}{1+q_m}}& \leq \left(q_{m-1}+bq_{m}+b(2Kq_{m})^{\frac{1+q_{m}}{1+q_{m-1}}}\right) ^{\frac{1}{1+q_{m}}}\\
&<\left(\left(\left(2Kb+b+1\right)q_m\right)^{\frac{1+q_m}{1+q_{m-1}}}\right)^{\frac{1}{1+q_m}}\\
&<\left(4Kbq_m\right)^{\frac{1}{1+q_{m-1}}}\leq C(4Kb)^{\sum_{k=N}^{m-1} \frac{1}{1+q_k}}\end{align*}
where we have used the induction hypothesis to obtain the upper bound 
\[q_m^{\frac{1}{1+q_{m-1}}}\leq C(4Kb)^{\sum_{n=N}^{m-2}\frac{1}{1+q_k}}.\]
 The second inequality of \eqref{eq:induct-2} is thus satisfied for $n=m+1$. Finally we have
\begin{align*}q_{m+1}^{\frac{1}{1+q_m}} &=\left(bq_m\left\lceil q_{m}^{-1}(2Kq_{m})^{\frac{1+q_{m}}{1+q_{m-1}}}\right\rceil + q_{m-1}\right)^{\frac{1}{1+q_m}}\\
&>b^{\frac{1}{1+q_m}} (2Kq_m)^{\frac{1}{1+q_{m-1}}}\\
&>(2Kq_m)^{\frac{1}{1+q_{m-1}}}\end{align*}
so that the left-hand side of \eqref{eq:induct-2} is also satisfied for $n=m+1$. We conclude that the choice of $a_{m+1}$ given above provides the requisite extension of the finite sequence $(a_n)_{n=0}^m$, and by induction we deduce the existence of the desired infinite sequence $(a_n)_{n=0}^\infty$. For the remainder of the proof we fix the sequence $(a_n)_{n=0}^\infty$ defined in the preceding manner.

Let us show that the number $\theta:=[a_0;a_1,a_2,\ldots]$ thus constructed has the properties required in the statement of the lemma. By \cite[Theorem 9]{Kh97} we have for every $n \geq 0$
\begin{equation}\label{eq:libel}\left|q_n(\theta-a_0) -p_n\right|<\frac{1}{q_{n+1}}.\end{equation}
If $n \geq N$ then necessarily $q_{n+1}\geq q_2=a_2q_1+1\geq 2$, so \eqref{eq:induct-1} and \eqref{eq:libel} together imply that the nearest integer to $q_n\theta$ is congruent to $a$ modulo $b$. For $n \geq N$ we thus have
\begin{equation}\label{eq:bewgry2}\mathrm{dist}\left(q_n\theta,a+b\mathbb{Z}\right)^{\frac{1}{1+q_n}} =\left|q_n(\theta-a_0) -p_n\right|^{\frac{1}{1+q_n}}<\left(\frac{1}{q_{n+1}}\right)^{\frac{1}{1+q_n}}<\left(\frac{1}{2Kq_{n}}\right)^{\frac{1}{1+q_{n-1}}}\end{equation}
using \eqref{eq:induct-2}.

We now make the following claim: if $k \geq 0$ is any integer, and $n \geq -1$ is the unique integer such that $q_n\leq k <q_{n+1}$, then
\[\mathrm{dist}(k\theta,a+b\Z)>\frac{1}{2q_{n+1}}.\]
We consider several cases. Firstly, if $k=0$ then $n=-1$ and we have
\[\mathrm{dist}(k\theta,a+b\Z)=\min\{|a|,|b-a|\}\geq 1>\frac{1}{2}= \frac{1}{2q_{n+1}}\]
as required. Suppose next that $k \geq 1$. Let $\ell \in \Z$ such that $\mathrm{dist}(k\theta,a+b\Z)=|k\theta-\ell|$, and let $d\geq 1$ denote the greatest common divisor of $k$ and $|\ell|$. Let $k=k'd$, $\ell=\ell'd$ where $k'$ and $\ell'$ are coprime, and let $m \geq 0$ such that $q_m \leq k'< q_{m+1}$. If $k'> q_m$ then by \cite[Theorem 19]{Kh97} we have $|k'\theta-\ell'|\geq 1/2k'$ and therefore
\[\mathrm{dist}(k\theta,a+b\Z) =|k\theta-\ell|=d|k'\theta-\ell'|\geq \frac{d}{2k'}\geq \frac{1}{2k}>\frac{1}{2q_{n+1}}\]
as required. If on the other hand $k'=q_m$ then since $k'\geq 1$ we have $q_{m+1}>1$ and therefore $q_{m+1} \geq 2$. It follows using \eqref{eq:libel} that the nearest integer to $q_m\theta$ is $p_m+a_0q_m$. By \cite[Theorem 13]{Kh97} we may therefore deduce
\[\mathrm{dist}(k\theta,a+b\Z) =|k\theta-\ell|=d|q_m\theta-\ell'|\geq d|q_m(\theta-a_0)-p_m|>\frac{d}{q_{m+1}+q_m}>\frac{1}{2q_{n+1}}\]
and this completes the proof of the claim.

We may now prove that $\theta$ has the properties required in the statement of the lemma. Let $k \geq 0$ be an integer, and let $n \geq -1$ such that $q_n \leq k <q_{n+1}$. We will show that there exists $q \geq \max\{q_{N},k+1\}$ such that
\[\mathrm{dist}(k\theta,a+b\Z)^{\frac{1}{1+k}}>K^{\frac{1}{1+k}} \mathrm{dist}(q\theta,a+b\Z)^{\frac{1}{1+q}}.\]
We consider two cases. If $n<N$, then by the previous claim together with \eqref{eq:bewgry} we have
\begin{align*}\mathrm{dist}(k\theta,a+b\Z)^{\frac{1}{1+k}}>\left(\frac{1}{2q_{n+1}}\right)^{\frac{1}{1+k}}&\geq \frac{1}{2q_{N}}\\
&>K\left(\frac{1}{q_{N+1}}\right)^{\frac{1}{1+q_{N}}}\geq K^{\frac{1}{1+k}}\left(\frac{1}{q_{N+1}}\right)^{\frac{1}{1+q_{N}}}\end{align*}
and therefore
\[\mathrm{dist}(k\theta,a+b\Z)^{\frac{1}{1+k}}>K^{\frac{1}{1+k}}\left(\frac{1}{q_{N+1}}\right)^{\frac{1}{1+q_{N}}}>K^{\frac{1}{1+k}}\mathrm{dist}(q_N\theta,a+b\Z)^{\frac{1}{1+q_N}}\]
where we have used \eqref{eq:bewgry2}. If on the other hand $n \geq N$, then by the previous claim together with \eqref{eq:bewgry2}
\begin{align*}\mathrm{dist}(k\theta,a+b\Z)^{\frac{1}{1+k}}&>\left(\frac{1}{2q_{n+1}}\right)^{\frac{1}{1+k}}\\
&\geq K^{\frac{1}{1+k}}\left(\frac{1}{2Kq_{n+1}}\right)^{\frac{1}{1+q_n}}\\
&>K^{\frac{1}{1+k}}\mathrm{dist}(q_{n+1}\theta,a+b\Z)^{\frac{1}{1+q_{n+1}}}.\end{align*}
This establishes the first of the two properties required in the statement of the lemma. It remains to prove \eqref{eq:second-lemma-bit}. We note that the preceding argument implies the inequality
\[\inf_{k \geq 0}\mathrm{dist}(k\theta,a+b\Z)^{\frac{1}{1+k}} \geq \inf_{n \geq N} \left(\frac{1}{2q_{n+1}}\right)^{\frac{1}{1+q_{n}}},\]
so to establish  \eqref{eq:second-lemma-bit} it is sufficient to prove the inequality
\[\inf_{n \geq N}\left(\frac{1}{q_{n+1}}\right)^{\frac{1}{1+q_n}}>0.\]
 By \cite[Theorem 12]{Kh97} we have $q_n \geq 2^{\frac{n-1}{2}}$ for every $n \geq 0$ so in particular
 \[\sum_{k=N}^\infty \frac{1}{1+q_k} <\sum_{k=1}^\infty \frac{1}{1+2^{\frac{k-1}{2}}}<\sum_{k=0}^\infty 2^{-\frac{k}{2}}<4.\]
 If $n\geq N$ then by \eqref{eq:induct-2}
\[\left(\frac{1}{q_{n+1}}\right)^{\frac{1}{q_{n}}} \geq C^{-1}\left(4Kb\right)^{-\sum_{k=N}^{n-1}\frac{1}{1+q_k}}>C^{-1}\left(4Kb\right)^{-\sum_{k=N}^{\infty}\frac{1}{1+q_k}}>C^{-1}(4Kb)^{-4}>0\]
and since the last two terms do not depend on $n$,  the infimum in question is nonzero. This completes the proof of \eqref{eq:second-lemma-bit} and hence completes the proof of the lemma.\end{proof}

We may now give the proof of Proposition \ref{pr:ctdfrn}, and deduce from it the result which will be used in the following section. 
\begin{proof}[Proof of Proposition \ref{pr:ctdfrn}]
Let $\tilde{\alpha},\tilde{\theta} \in \mathbb{R}$, and let $\varepsilon>0$; we will define $\alpha,\theta$ such that $|\alpha-\tilde\alpha|$ and $|\theta -\tilde\theta|$ are smaller than $\varepsilon$ and such that $(H_\alpha,R_\theta)$ has positive lower spectral radius but lacks the lower finiteness property. Since $R_\phi=R_{\phi+2\pi n}$ for every $\phi \in \R$ and $n \in \Z$ we may freely assume that $\tilde{\theta}\in(-\frac{\pi}{2},\frac{\pi}{2}]$, and by making an arbitrarily small adjustment to $\tilde\theta$ if necessary we may also assume that $\pi^{-1}\tilde\theta$ is irrational.
 
    Since the set of rationals with prime denominator is dense in $\R$, and the function $\tan \colon (-\frac{\pi}{2},\frac{\pi}{2})\to\R$ is a homeomorphism, we may choose $\alpha \in \R$ such that $|\alpha-\tilde{\alpha}|<\varepsilon$ and $\lambda\alpha^{-1} =-\tan(\frac{\pi a}{b})$ for some nonzero integer $a$ and prime natural number $b\geq 5$ such that $\frac{a}{b} \in (-\frac{1}{2},\frac{1}{2})$.  By Lemma \ref{le:ratio} we may choose a constant $C_0>1$ such that
\begin{equation}\label{eq:nice-1}C_0\left|\lambda \cos \phi + \alpha\sin \phi\right| \geq \mathrm{dist}\left(\phi,\frac{\pi a}{b}+\pi\mathbb{Z}\right) \geq C_0^{-1}\left|\lambda \cos \phi + \alpha\sin \phi\right| \end{equation}
for every $\phi \in \mathbb{R}$.  By Lemma \ref{le:main-lemma-ctd-frn} there exist an infinite sequence of integers $(a_n)_{n=0}^\infty$ with $a_n \geq 1$ for all $n \geq 1$, and a real number $\delta>0$, such that the irrational number $\vartheta:=[a_0;a_1,\ldots,] \in (-\frac{b}{2},\frac{b}{2})$ satisfies $|\vartheta-b\pi^{-1}\tilde{\theta}|<\varepsilon b\pi^{-1}$, and such that the inequality
\begin{equation}\label{eq:nice-2}\mathrm{dist}(n\vartheta,a+b\Z)^{\frac{1}{1+n}}>\left(\frac{C_0b}{\pi}\right)^{\frac{2}{1+n}}\left(\inf_{m> n}\mathrm{dist}(m\vartheta,a+b\Z)^{\frac{1}{1+m}}\right)\geq \delta>0\end{equation}
holds for every integer $n \geq 0$, where we have used $C_0b/\pi \geq 5C_0/\pi>C_0>1$. 

Define $\theta:=\frac{\pi\vartheta}{b} \in (-\frac{\pi}{2},\frac{\pi}{2})$ so that $|\theta-\tilde\theta|<\varepsilon$. 
Since $\vartheta$ is irrational, $\theta$ is not a rational multiple of $\pi$, so we have
\[\inf_{n \geq 0}\rho\left(H_\alpha R_\theta^n\right)^{\frac{1}{1+n}}=\inf_{n \geq 0}|\lambda \cos n\theta + \alpha \sin n\theta|^{\frac{1}{1+n}}<1=\rho(R_\theta)\]
using Lemmas \ref{le:explicittrace} and \ref{le:ae-lemma-1}. If $(H_\alpha,R_\theta)$ has the lower finiteness property then necessarily $\underline{\varrho}(H_\alpha,R_\theta)=\rho(H_\alpha R_\theta^n)^{1/(n+1)}$ for some integer $n \geq 0$, since otherwise Lemma \ref{le:newformula} together with the previous inequality yields a contradiction. To show that $(H_\alpha,R_\theta)$ does not have the lower finiteness property it is therefore sufficient to show that the infimum
\[\inf_{n \geq 0}\rho\left(H_\alpha R_\theta^n\right)^{\frac{1}{n+1}}\]
is not attained. By Lemma \ref{le:formula}, to show that $\underline{\varrho}(H_\alpha,R_\theta)>0$  it is sufficient to show that the same infimum is nonzero.

Let us first prove that this infimum is not attained. Given $n \geq 0$, using Lemma \ref{le:explicittrace}, \eqref{eq:nice-1} and \eqref{eq:nice-2} we may find $m>n$ such that 
\begin{align*}\rho\left(H_{\alpha}R_{\theta}^n\right)^{\frac{1}{1+n}}&=\left|\lambda \cos n\theta+\alpha_0\sin n\theta \right|^{\frac{1}{1+n}}\\
&\geq C_0^{-\frac{1}{1+n}}\mathrm{dist}\left(n\theta,\frac{\pi a}{b}+\pi\mathbb{Z}\right)^{\frac{1}{1+n}}\\
&= \left(\frac{\pi}{C_0b}\right)^{\frac{1}{1+n}} \mathrm{dist}(n\vartheta,a+b\mathbb{Z})^{\frac{1}{1+n}}\\
&>  \left(\frac{C_0b}{\pi}\right)^{\frac{1}{1+n}}\mathrm{dist}(m\vartheta,a+b\mathbb{Z})^{\frac{1}{1+m}}\\
&>  \left(\frac{C_0b}{\pi}\right)^{\frac{1}{1+m}}\mathrm{dist}(m\vartheta,a+b\mathbb{Z})^{\frac{1}{1+m}}\\
&=C_0^{\frac{1}{1+m}}\mathrm{dist}\left(m\theta,\frac{\pi a}{b}+\pi\mathbb{Z}\right)^{\frac{1}{1+m}}\\
&\geq \left|\lambda \cos m\theta+\alpha_0 \sin m\theta\right|^{\frac{1}{1+m}}=\rho\left(H_{\alpha}R_{\theta}^m\right)^{\frac{1}{1+m}}\end{align*}
proving our assertion. To see that $\underline{\varrho}(H_\alpha,R_\theta)$ is nonzero we note that by a similar chain of inequalities
\[\underline{\varrho}(H_\alpha,R_\theta)=\inf_{n \geq 0}\rho(H_\alpha,R_\theta^n)^{\frac{1}{1+n}}\geq \inf_{n \geq 0} \left(\frac{\pi}{C_0b}\right)^{\frac{1}{1+n}}\mathrm{dist}(n\vartheta,a+b\Z)^{\frac{1}{1+n}}\geq \frac{\pi\delta}{C_0b}>0.\]
The proof of the proposition is complete.
\end{proof}

We now easily deduce the following result which will be used in the proof of Theorem \ref{th:only}.

\begin{corollary}\label{co:dense-ctd-frn}
The set of all $(H,R) \in \mathcal{P}\times\mathcal{E}$ such that $\underline{\varrho}(H,R)>0$ and such that $(H,R)$ does not have the lower finiteness property is a dense subset of $\mathcal{P}\times\mathcal{E}$.
\end{corollary}
\begin{proof}
We argue similarly to the proof of Proposition \ref{pr:full-measure-2}. Let us define $\Xi \subseteq SL_2^\pm(\R)\times (\mathbb{R}\setminus\{0\})^2\times \R^2$ to be the set of all $(A,\gamma,\lambda,\alpha,\theta)$ such that the pair
\[\left\{\left(\begin{array}{cc}\lambda & \alpha \\0&0\end{array}\right),\left(\begin{array}{cc}\cos \theta & -\sin \theta\\\sin\theta&\cos \theta\end{array}\right)\right\}\]
has nonzero lower spectral radius but fails to have the lower finiteness property.  By Proposition \ref{pr:ctdfrn} this set is a dense subset of $SL_2^\pm(\R)\times(\mathbb{R}\setminus\{0\})^2\times \R^2$. If we define $\Phi \colon SL_2^\pm(\R)\times(\mathbb{R}\setminus\{0\})^2\times \R^2 \to \mathcal{P}\times\mathcal{E}$ by
\[\Phi(A,\gamma,\lambda,\alpha,\theta) :=\gamma\left(A\left(\begin{array}{cc}\lambda & \alpha \\0&0\end{array}\right)A^{-1},A \left(\begin{array}{cc}\cos \theta & -\sin \theta\\\sin\theta&\cos \theta\end{array}\right)A^{-1}\right)\]
then by Lemma \ref{le:similar} the function $\Phi$ is surjective. Clearly $\Phi$ is also continuous. It is straightforward to see that if $(H,R) \in \Phi(\Xi)$ then $(H,R)$ has nonzero lower spectral radius and does not have the lower finiteness property, and since $\Phi$ is continuous and surjective, $\Phi(\Xi)$ is dense in $\mathcal{P}\times\mathcal{E}$.\end{proof}

\section{Proof of Theorem \ref{th:only}}

We first claim that $\underline{\varrho} \colon \mathcal{P}\times\mathcal{E}\to\R$ is continuous at $(H,R)$ if and only if $\underline{\varrho}(H,R)=0$. On the one hand, by Proposition \ref{pr:full-measure-2} the set of all $(H,R) \in \mathcal{P}\times\mathcal{E}$ such that $\underline{\varrho}(H,R)>0$ has full Lebesgue measure, and in particular this set is dense. On the other hand by Lemma \ref{le:dense-zeros} the set of all $(H,R) \in \mathcal{P}\times \mathcal{E}$ such that $\underline{\varrho}(H,R)=0$ is also dense. In particular if $(H,R) \in \mathcal{P}\times\mathcal{E}$ and $\underline{\varrho}(H,R)>0$ then $(H,R)$ is a point of discontinuity of $\underline{\varrho}$; but by Lemma \ref{le:zero-cont}, if $(H,R) \in \mathcal{P}\times\mathcal{E}$ and $\underline{\varrho}(H,R)=0$ then $\underline{\varrho}$ is continuous at $(H,R)$, which proves the claim. We deduce from the claim that $\mathcal{P}\times\mathcal{E}$ is equal to the disjoint union of the four sets $\mathcal{U}_i$.

By Lemma \ref{le:dense-zeros}, Corollary \ref{co:dense-ctd-frn} and the above considerations, the sets $\mathcal{U}_2$ and $\mathcal{U}_3$ are dense in $\mathcal{P}\times\mathcal{E}$. For every $n \geq 1$ the set
\[X_n:=\mathrm{Int }\left\{(H,R) \in \mathcal{P}\times\mathcal{E}\colon \underline{\varrho}(H,R)<\frac{1}{n}\right\}\]
is by definition open, and by Lemma \ref{le:zero-cont} this set contains every zero of $\underline{\varrho}$ in $\mathcal{P}\times\mathcal{E}$. By Lemma \ref{le:dense-zeros} each $X_n$ is dense. On the other hand, for each $n_1,\ldots,n_k \geq 0$ and $m_1,\ldots,m_{k} \geq 0$, by Lemma \ref{le:notmanyzeros} the set
\[Y_{(n_1,\ldots,n_k;m_1,\ldots,m_k)}:=\left\{(H,R) \in \mathcal{P}\times\mathcal{E} \colon H^{n_k}R^{m_k}\cdots H^{n_1}R^{m_1} \neq 0\right\}\]
is open and dense. If there exists a product $H^{n_k}R^{m_k}\cdots H^{n_1}R^{m_1}$ with zero spectral radius then by the Cayley-Hamilton theeorem the product $(H^{n_k}R^{m_k}\cdots H^{n_1}R^{m_1} )^2$ must be zero, so the set
\[Y:=\left\{(H,R) \in \mathcal{P}\times\mathcal{E}\colon \text{ no product of }H\text{ and }R\text { is zero}\right\}\]
is equal to the set
\[\bigcap_{k=1}^\infty \bigcap_{\substack{n_1,\ldots,n_k \geq0\\m_1,\ldots,m_k \geq 0}} Y_{(n_1,\ldots,n_k,m_1,\ldots,m_k)}\]
which is the intersection of countably many open dense sets. Since furthermore
\[\mathcal{U}_1= Y \cap \bigcap_{n=1}^\infty X_n\]
we conclude that $\mathcal{U}_1$ is equal to the intersection of countably many open dense subsets of $\mathcal{P}\times\mathcal{E}$. In particular it is dense by Baire's theorem. We finally note that the set $\mathcal{U}_4$ has full Lebesgue measure by Proposition \ref{pr:full-measure-2}, and this completes the proof of Theorem \ref{th:only}.

\section{Acknowledgements}

This research was supported by EPSRC grant  EP/L026953/1.

\bibliographystyle{siam}
\bibliography{beanspirit}
\end{document}